\documentclass[11pt,letter]{amsart}
\usepackage[utf8]{inputenc}

\usepackage{amsmath,amsthm,verbatim,amssymb,amsfonts,amscd, graphicx}
\usepackage{xcolor}
\usepackage{graphics}
\usepackage{mathtools}

\usepackage{euscript, enumerate,enumitem}
\usepackage{url,hyperref}
\mathtoolsset{showonlyrefs}

\newcommand{\p}{\partial}

\newcommand{\la}{\langle}
\newcommand{\ra}{\rangle}

\newcommand{\e}{\varepsilon}

\newcommand{\eps}{\varepsilon}
\newcommand{\kr}{{\mathfrak r}}
\newcommand{\kb}{{\mathfrak b}}
\newcommand{\bnabla}{ \overline{\nabla}}

\newcommand{\be}{\begin{equation}}
\newcommand{\ba}{\begin{aligned}}
\newcommand{\bee}{\begin{equation*}}
\newcommand{\ee}{\end{equation}}
\newcommand{\ea}{\end{aligned}}
\newcommand{\eee}{\end{equation*}}
\newcommand{\bea}{\begin{equation} \begin{aligned} }
\newcommand{\eea}{\end{aligned}\end{equation} }

\theoremstyle{plain}
\newtheorem{theorem}{Theorem}[section]

\newtheorem{lemma}[theorem]{Lemma}

\newtheorem{prop}[theorem]{Proposition}
\newtheorem{conjecture}[theorem]{Conjecture}

\theoremstyle{remark}

\newtheorem{remark}[theorem]{Remark}
\newtheorem{example}[theorem]{Example}

\theoremstyle{definition}
\newtheorem{definition}[theorem]{Definition}

\numberwithin{equation}{section}

\title[Hypersurfaces translating by powers of Gauss curvature]{On existence of hypersurfaces translating by powers of Gauss curvature}
\author{Beomjun Choi}
\address{Beomjun Choi: Department of Mathematics, POSTECH, Pohang, Gyeongbuk, Republic of Korea}
\email{bchoi@postech.ac.kr}

\date{}

\begin{document}

\begin{abstract} In this paper we construct complete convex hypersurfaces in $\mathbb R^{n+1}$ which translate under the flow by powers $\alpha \in (0, \frac1{n+2})$ of the Gauss curvature. The level set of each solution is asymptotic to a shrinking soliton for the flow by power $\frac \alpha {1-\alpha}$ of the Gauss curvature in $\mathbb R^n$. For example, our construction reveals the existence of translators whose level set converges to the sphere, simplex, hypercube and so on. The translating solitons exist as a family whose parameters correspond to Jacobi fields, solutions to linearized equation around the asymptotic profile.\end{abstract}
 
\maketitle 
\tableofcontents

\section{Introduction}

One-parameter family of convex embedded hypersurfaces $\Sigma_t$ in $\mathbb{R}^{n+1}$ is a solution to the $\alpha$-Gauss curvature flow if 
\[\partial_t X = - K^\alpha \nu .\]
Here, $X:[0,T]\times M \to \mathbb{R}^{n+1}$ is the family of embeddings with $X(t, M)=\Sigma_t$, $K$ is the gauss curvature at $X\in M_t$, and $\nu$ is the outward unit normal vector at $X\in M_t$. 

There are two types of special solutions which attract attentions among researchers; the shrinking soliton and the translating solitons. In this paper, we aim to reveal a connection between two solutions. More precisely, we construct translating solitons in $\mathbb{R}^{n+1}$ whose level sets are asymptotic to shrinking solitons in $\mathbb{R}^n$.

A shrinking soliton refers a solution which homothetically shrinks to the origin. Upto a rescaling, shrinking soliton satisfies $ \la X ,\nu \ra = K^\alpha $. Indeed if $\Sigma_0$ satisfies the equation, then $\Sigma_t$ shrinks by $\Sigma_t=(1- (1+n\alpha) t)^{\frac{1}{1+n\alpha}}\Sigma_0.$ Shrinking solitons play an important role as they models singularity formations of flows starting from closed hypersurface for affine-super-critical case $\alpha\in(\frac{1}{n+2},\infty)$ \cite{andrews2016flow}\cite{guan2017entropy}. Later \cite{brendle2017asymptotic} showed the round sphere is the only closed shrinking soliton in this case so the asymptotic behavior of closed solutions are now understood. In the affine-critical $\alpha=\frac{1}{n+2}$, the solutions have affine invariance and  \cite{Ca72} showed shrinking solitons are ellipsoids. Later, the convergence of $\frac{1}{n+2}$-GCF to an ellipsoid was shown by \cite{AC}. Shrinking solitons for sub-affine-critical $\alpha\in (0,\frac{1}{n+2})$ are not exactly classified except one dimensional curves in $\mathbb{R}^2$ by \cite{andrews2003classification}. Such a classification is a big open question. See \cite{andrews2000motion} (or see Theorem \ref{theorem-andrews2000}) for constructions of some shrinking solitons in higher dimensions. Next, a translating soliton refers a solution which moves in a constant velocity. Upto a rescaling and rotation, such a solution satisfies $ \la \nu, -e_{n+1}\ra = K^\alpha$ and in this case the solution moves by $\Sigma_t= \Sigma_0 +te_{n+1}$. For other geometric flows such as the Ricci flow and the mean curvature flow, translating solitons also appear as singularity models for degenerate neck-pinches, or more generally type-II singularities. Since solutions to sub-affine-critical GCF is expected to degenerate generically (see this is true when $n=1$ \cite{andrews2002non}), there is huge possibility that translating solitons may appear as singularity models for sub-affine-critical flows. Moreover, they have stability in the sense that near by non-compact solutions converge to translating solitons as time goes to infinity \cite{CHOI2022108207}\cite{choi2018convergence}.

When translating solitons are written as graphs  $x_{n+1}=u(x)$, the height function $u(x)$ solve the Monge-Amp\`ere equation \[\det D^2 u = (1+|Du|^2)^{\frac{n+2}{2}-\frac{1}{2\alpha} }.\] The existence and classification of global solutions to those equations have been classical questions in the theory of Monge-Amp\`ere equation. In the affine-critical case, the equation becomes simple $\det D^2u =1$ and it is a classical result that entire convex solutions are quadratic polynomials \cite{nitsche1956elementary, jorgens1954losungen, calabi1958improper, cheng1986complete, caffarelli1995topics}; see also \cite{caffarelli2003extension, caffarelli2004liouville, jin2014liouville, jin2016solutions, li2019monge} for more Liouville theory of Monge-Amp\`ere equations.  When $\alpha>\frac{1}{2}$, the translating solitons are always convex graphs on bounded domains and \cite{U98GCFsoliton} showed each convex bounded domain admits unique translator which is a graph on this domain. For remaining ranges of alpha, $\alpha \in (0,1/2]$ and $\alpha \neq \frac{1}{n+2}$, \cite{JW} showed they are entire graphs for $\alpha<\frac{1}{n+1}$ and constructed infinitely many translators which are not translation or rotation of the others, but the classification is still an open question. Recently, for the case of surface $n=2$, the author, K.Choi and S.Kim \cite{CCK} classified all possible translating solitons for the sub-affine-critical range $\alpha \in (0,\frac{1}4)$. In  \cite{CCK}, constructed solutions are asymptotic to a $\frac{1}{1-2\alpha}$-homogeneous function, and moreover if the level curves $u(x)=l$, after a rescaling, converges to a shrinking soliton to $\frac{\alpha}{1-\alpha}$-GCF (of curves). Using the classification of shrinking curves in \cite{andrews2003classification}, we could construct all possible candidates, and conversely  every translator is exactly one of these candidates. We propose a similar classification conjecture for higher dimension in Conjecture \ref{conjecture}. In this work, we modify the construction part of the previous work \cite{CCK} and showed the existence of translators asymptotic to given homogeneous profile corresponding to a shrinking soliton in $\mathbb{R}^{n}$. The converse part, where  every soliton is shown to be asymptotic to a homogeneous profile and it is exactly one of solution already constructed, faces new challenges that we want to resolve in future research.

Our main existence result is Theorem \ref{thm-mainexistence}. In higher dimension, as graphs our translators are asymptotic to $\frac{1}{\sigma}$-homogeneous function where $\sigma$ is given in \eqref{eq-sigma A}. Since the level sets are asymptotic to shrinking solitons, we have benefits working with the representation of hypersurface by support functions at different levels $x_{n+1}=l$ (see \eqref{eq-Sltheta} for precise definition). If $S(l,\theta)$ denotes the support function of translating soliton at the height $l$ and $h(\theta)$ denotes the support function of asymptotic shrinking soliton, the solution is asymptotic to $S\approx  A l^{\sigma} h(\theta) $ as $l\to \infty$. In Theorem \ref{thm-mainexistence}, we construct family of translating solitons for each given fixed limit shrinking soliton. This family corresponds to solutions to linearized equation to the soliton equation around $Al^\sigma h(\theta)$. We call them Jacobi fields and they can be written as separation of variable form $l^{\beta^\pm_i } \varphi_i(\theta)$ where $\varphi_i(\theta)$ are eigenfunctions for the linearization of Gauss curvature around the shrinking soliton. In Section \ref{section-prelim}, we prepare necessary function spaces and spectral set-up for main analysis. We also introduce the support function and express the soliton equation in terms of this and its perturbation around homogeneous limiting profile. In Section \ref{section-main}, we first state the main Theorem \ref{thm-mainexistence}.  The proof is given at the end of Section \ref{section-main}. It basically follows similar lines of argument in \cite[Section 3]{CCK}. Here, we focus on presenting crucial ideas and methods toward the construction. With this purpose, we omit precise details of computation if they are identical to those presented in \cite{CCK}. Between the statement and the proof of Theorem \ref{thm-mainexistence}, we explain remarks and consequences of this construction. Since the shrinking solitons are not classified in higher dimension, we explain known existence result for non-radial shrinking solitons based on \cite{andrews2000motion}. Those shrinking solitons with the symmetry of simplex or hypercube mentioned in the abstract arise as bifurcations of spherical harmonics from radial shrinking soliton as $\alpha$ cross (from above to below) certain threshold exponents. 
  
 \section{Preliminaries} \label{section-prelim}
\begin{definition}[Support function at level $l$]
Let $\Sigma$ be a complete convex hypersurface in $\mathbb{R}^{n+1}$. If $\Sigma$ has non-empty level set at level $x_{n+1}=l$,  we denote the support function of this level set by $S(l,\theta)$. More precisely,
\be \label{eq-Sltheta} S (l,\theta  ) = \sup_{p \in \Sigma \cap \{x_{n+1}  = l\}} \langle p, (\theta ,0) \rangle, \ee 
for $\theta  \in \mathbb{S}^{n-1}$. Here, we view $(\theta ,0)$ as a vector in $\mathbb{R}^{n+1}$. We simply say $S(l,\theta )$ is the support function of $\Sigma$ (at level $l$). 	
\end{definition}

Throughout this paper, $\bar g$ and $\overline{\nabla}$ denotes the standard round metric on $\mathbb{S}^{n-1}$ and the corresponding Levi-Civita connection.  A convex complete hypersurface $\Sigma$ in $\mathbb{R}^{n+1}$ in is a $\beta$-translator if it is a smooth translating soliton for the $\beta$-GCF moving in $e_{n+1}$-direction with unit speed \[ K^\beta = \la  e_{n+1},-\nu \ra .\] A convex compact hypersurface $\Gamma$ in $\mathbb{R}^n$ is a $\beta$-shrinker if it is a smooth shrinking soliton for the $\beta$-GCF which satisfies \[K^\beta = \la F, \nu  \ra  .\]  Since we construct $\alpha$-translators and those solutions are asymptotic to $\frac{\alpha}{1-\alpha}$-shrinker, we will often omit the exponents when there is no confusion. i.e. a translator refers a $\alpha$-translator in $\mathbb{R}^{n+1}$ and a shrinker refers a $\frac{\alpha}{1-\alpha}$-shrinker in $\mathbb{R}^n$.

Suppose a closed smooth strictly convex hypersurface $\Gamma$ in $\mathbb{R}^n$ is parametrized by the outward normal vectors $\theta \in \mathbb{S}^{n-1}$. If the support function is $S(\theta)$, then the second fundamental form is given by $h_{ij} = \overline{\nabla}^2_{ij} S+ S \bar g_{ij}$ and the principle radii are given by the eigenvalues of $h_{ij}$ with respect to $\bar g_{ij}$ (see, for instance, \cite[Chapter 3.1]{MR1931534}.) With this in mind, it is convenient to define the following operators:

\begin{definition}
For scalar function $f$ on $\mathbb{S}^{n-1}$, we define the $(0,2)$-tensor $\kr_{ij}[f]$ on the sphere by 
${\mathfrak r} _{ij}[f] := \bnabla^2_{ij} f + f \bar g_{ij} $.
Moreover we denote the inverse of ${\mathfrak{r}}_{ij} [f]$ (if exists) by ${\mathfrak b}^{ij}[f] $. We denote the determinant with respect to $\bar g$, the product of the eigenvalues of $\kr_{ij}[h]$ with respect to $\bar g_{ij}$, by
\[\det _{\bar g} (\kr_{ij} [f]):= \det (\bar g^{ik} \kr _{kj}[f]).\]  
\end{definition}
Our preferred way of expressing the equation is to use the support function for each level set as in \eqref{eq-Sltheta}. The translator equation becomes the following. 
\begin{lemma}[Equation of translator in support function]Let $S(l,\cdot )$ represent the support function of a translator at level $l$. Then $S(l, \theta )$ is a solution to \be\label{eq-translatorsupport} S_{ll}+(1+S_l ^{2})^{\frac{n+2}{2} - \frac{1}{2\alpha}} S_l^{\frac1 \alpha}{\det_{\bar g}(  \kr _{ij}[S]) }   =0.\ee 
	Similarly, if $S(l,\cdot )$ represents the support function of a blow-down translator at level $l$, then it solves 
	\be\label{eq-blowdownsupport} S_{ll}+S_l^{\frac1 \alpha} {\det_{\bar g}(  \kr _{ij}[S]) }  =0. \ee  
\begin{proof}

Let $\Sigma\subset \mathbb{R}^{n+1}$ be a translator (or blow-down translator). Then we may parametrize $\Sigma$ by the height $l$ and the outward normal of level set $\{\Sigma \cap \{x_{n+1}=l\}$, say $\theta  \in \mathbb{S}^{n-1}$. i.e. $F(l,\theta )$ is the unique point in $\Sigma \cap\{x_{n+1}=l\}$ whose the outward normal on  $\Sigma \cap\{x_{n+1}=l\}$ (viewing it as a hypersurface in $\mathbb{R}^{n}$) is $\theta $. 

Let us choose a local coordinate chart $\{x^i\}$ on $\mathbb{S}^{n-1}$ and the standard Levi-Civita connection on $\mathbb{S}^{n-1}$ by $\bnabla$. Then
\[\la F ,\theta  \ra = S \text{ and } \bnabla_i S = \la F, \theta _i \ra \] (here $\theta _j = \partial_j \theta $)
 implies \[F(l,\theta )= S \theta  + \bnabla^jS \theta _j + l e_{n+1}.\]
  
 We may compute the derivatives 
 \bea \partial_i F &= S \theta  _i + { \bnabla}_i \bnabla^j S \theta  _j \\
\partial_l F &= S_l \theta   + \bnabla^j  S_l  \theta _j   + {e}_{n+1} .\eea 
Using this, we compute the outward normal (to $\Sigma$) 
$ \nu = -\frac{(-\theta , S_l)}{(1+S_l^2)^{\frac 12}}$
and the second derivatives 
\bea  \partial^2_{ll}F&=S_{ll} \theta  + \bnabla^j   S_{ll} \theta  _j \\ 
\partial^2_{il}F&=S_{l} \theta _i + \bnabla_i \bnabla^j   S_{l } \theta _j \\ 
  \la \partial^2_{ij} F,\theta  \ra &= S \la \partial_j w_i , w\ra + { \bnabla}_i \bnabla^k S  \la \partial_j \theta _k ,\theta  \ra \\
&=   - S \bar g_{ij} - \bnabla_i \bnabla_j S . \eea 
We infer the the second fundamental form
\bea h_{ij}&=\la -\partial^2_{ij} F,\nu  \ra  = [\bnabla^2_{ij} S+ S \bar g_{ij}](1+S_l^2)^{-\frac12 }  \\
h_{il} &=\la -\partial^2_{il} F,\nu  \ra=0 \\ 
h_{ll} &= \la -\partial^2_{ll } F,\nu  \ra= -S_{ll}(1+S_l^2)^{-\frac12 }  \eea
and the metric 
\bea  g_{ij} &= \la \partial_i F,\partial _j F \ra =  \kr_{ik}[S]\bar g^{km} \kr_{mj}[S] \\
 g_{il} & = \la \partial_i F, \partial_ l F \ra = \kr_{ik}[S]\bar g ^{km} \bnabla_m S_l    \\
 g_{ll} & = \la \partial_l F, \partial_ l F \ra = S_l^2 + \bnabla_k S_l \bar g^{km} \bnabla_m S_l +1\eea 
from above. The Gauss curvature 
$K= \det_{g} h_{\alpha \beta }= \frac{\det h_{\alpha \beta }}{\det g_{\alpha \beta }}$ follows from 
\[\det h_{\alpha \beta } = -\det (\kr_{ij}[S]) S_{ll} (1+S_l^2)^{-\frac n2 }\]
and
\[\det g_{\alpha \beta } = (\det  \kr_{ij} [S])^2 (\det \bar g^{ij})(1+S_l^2). \]

Finally, if $\Sigma$ is a translator, then  $K= \la e_{n+1} , -\nu \ra ^{\frac 1 \alpha}$ and hence  
\[ \frac{ -S_{ll}  (1+S_l^2)^{-\frac{n+2}2 } }{ \det_{\bar g} \kr_{ij} [S]  }=  \bigg ( \frac{ S_l}{( 1+S_l^2)^{\frac 12}} \bigg  )^{\frac{1}\alpha}   .\]
If $\Sigma$ is a blow-down translator, $\det D^2u = |Du|^{n+2-\frac 1\alpha}$, then since $K= \frac{\det D^2 u} {(1+|Du|^2)^{\frac{n+2}2}}$ and $|Du|= S_l^{-1}$, we obtain
\[\frac { -S_{ll}  }{ \det_{\bar g} \kr_{ij} [S]  }= S_l ^{\frac{1}\alpha}.\]

\end{proof}

\end{lemma}

\begin{remark}
Let $\Gamma^{n-1}$ be a translator in $\mathbb{R}^n$. Then the support function $h(\theta )$, on $\theta  \in \mathbb{S}^{n-1}$, satisfies $ K_\Gamma ^{\frac{\alpha}{1-\alpha}}(\theta )=h(\theta )$. Here, $K_\Gamma (\theta )$ is the Gauss curvature of $\Gamma$ at the point whose outward normal is $\theta \in \mathbb{S}^{n-1}$. Since \be K_\Gamma =  [ {\det _ {\bar g} ( \kr _{ij}[h]   } )]^{-1} = h^{\frac{1}{\alpha} -1 },\ee we may use $h(\theta )$ to obtain a separation of variables solution to the blow-down translator equation in $\mathbb{R}^{n+1}$ $$S(l,\theta ) = A l^\sigma   h(\theta ) ,$$ where \be \label{eq-sigma A}\sigma  = \frac{1-2\alpha}{1+ \alpha(n-2)} \quad \text{     	and } \quad A =  \sigma ^{\frac{\alpha-1}{1+\alpha (n-2)}} (1-\sigma )^{\frac{\alpha}{1+\alpha(n-2)}}.\ee  This $\sigma$-homogeneous ansatz $A l^\sigma   h(\theta )$ will provide the leading asymptotics for the translators.  
\end{remark}

\begin{lemma}[Linearization of $K^{\frac{\alpha}{1-\alpha}}$] \label{lem-spectrum} Let $h(\theta )$ on $\theta \in \mathbb{S}^{n-1}$ be the support function of a smooth shrinker.  Let us define weighted $L^2$ space $$L^2_h= \{f \in L^2(\mathbb{S}^{n-1})\,:\, \int_{\mathbb{S}^{n-1}} f^2 K^{-\frac{1}{1-\alpha}} d\bar g=\int_{\mathbb{S}^{n-1}} f^2 \, (\det _{\bar g} \kr_{ij}[h])^{\frac 1{1-\alpha}} d\bar g \}$$ and associated inner product $\la\cdot ,\cdot \ra_h$.  Then the the linear operator $$Lf=  (\det_{\bar g} {\kr}_{ij}[h] )^{-\frac{\alpha}{1-\alpha }} (\kb^{ij}[h] ) (\bnabla^2_{ij}f + f \bar g_{ij})=h (\kb^{ij}[h] ) (\kr _{ij}[f])   $$ is self-adjoint with the inner product and discrete eigenfunctions diagonalize $L^2_{h}$. There is orthonormal basis $\{\varphi_{i}\}_{i=0}^\infty $ in $L^2_h$ such that $L \varphi_i = \lambda_i \varphi_i$ with non-increasing $\lambda_i$ which diverges to $-\infty$ as $i\to \infty$. 

\begin{proof} The self-adjointness of $L$ in $\la ,\ra _h$ is direct consequence of $$ \bnabla_k [ (\det_{\bar g} \kr _{ij}[h] ) \kb ^{kl } [h] ]=0.$$ By the regularity estimate for uniformly elliptic equation, $(L-\mu I)^{-1}$ is compact operator in $L^2_h$ which admits a spectral decomposition by orthonormal eigenfunctions $\{\varphi_i \}_{i=0}^\infty$. $L \varphi_i=\lambda_i  \varphi_i $ and $\lambda_i\to -\infty$ due to the compactness of $(L-\mu I)^{-1}$.   
	
\end{proof}

\end{lemma}

Note that $L h= h \,\kb ^{ij}[h] \,\kr _{ij}[h] = (n-1) h$ is the first eigenfunction since the function is signed. Next, we will identify other non-negative eigenfunctions. In the case the shrinker is radial $h\equiv 1$, the inner product becomes the standard $L^2$-inner product,  $Lf= \Delta_{\bar g} f+(n-1 )f$ and one can explicitly enumerate $\varphi_i$ and $\lambda_i$ using the spherical harmonics. For instance, $\lambda_1=\cdots =\lambda_{n}=0$, and corresponding eigenfunctions can be chosen as $\varphi_i = x_{i} /\Vert x_{i} \Vert_{L^2_h} $ where $x_k$ is the $k$-th coordinate function on $\mathbb{R}^n$ with its domain restricted to $\mathbb{S}^{n-1} \subset \mathbb{R}^n$. (See Example \ref{example-sphere} for more eigenvalues.) The same coordinate function still becomes an eigenfunction of eigenvalue $0$ for general shrinker $h$.  This is true since $L$ originates from the linearization of $K^{\frac{\alpha}{1-\alpha}}$ and $K(w)$ is invariant under the translation of hypersurface in $\mathbb{R}^{n}$. It is a usual standard that the first eigenvalue is denoted by $\lambda_1$, but here we start from $\lambda_0$ in order to make $\lambda_i$ correspond to eigenfunctions $x_i$, for $i=1,\ldots,n$. 

\begin{lemma}[c.f. see Lemma 5 \cite{andrews2002non}  for $n=2$] For any shrinker $h$, $h$ and coordinate functions $x_1, \ldots, x_n$ span all non-negative eigenfunctions to $L$.  
	\begin{proof}
The Brunn-Minkowski inequality for convex sets in $\mathbb{R}^{n}$ says the $n$-th square root of the area functional is convex with respect to Minkowski edition. For two convex sets $\Omega_i$, $i=0$ and $0$, we define $\Omega_t =\{ty +(1-t)x \,:\, y\in \Omega_1,\, x\in \Omega_0\}$, then $\frac{d^2}{dt^2} A[\Omega_t]^{\frac 1n} \le 0$ with the equality if and only if $\Omega_1 = c\Omega_0 + e$ for some $c>0$ and $e\in \mathbb{R}^n$, which is the case when $\Omega_i$ are scaled translates of each other. 

We test the inequality for the following convex sets. Let $\Omega_0$ be the region enclosed by shrinker with the support function $h$ and $\Omega_1$ be the region enclosed by closed hypersurface whose support function is $h+ \delta f$ where $f$ is an eigenfunction of $L$ and $\delta>0$ small number. Then one may check $\Omega_t$ has the support function 
$h_t= h+ t\delta f$ and the area $A[\Omega_t]= \frac{1}{n} \int_{S^{n-1}} h_t K^{-1} d\theta = \frac{1}{n}\int_{S^{n-1}} h_t  \det _{\bar g} \kr_{ij}[h_t]  d\theta. $
From this, we compute
\bea \frac{d}{dt} A[\Omega_t]&= \frac1n \int  \delta f (\det _{\bar g} \kr_{ij} [h_t])  + h_t (\det _{\bar g} \kr_{ij} [h_t] ) \kb^{ij}[h_t] \kr_{ij} [\delta f] \, d\theta \\
  &= \frac{1}{n} \int \delta f (\det _{\bar g} \kr_{ij} [h_t] )+ (n-1) (\det _{\bar g} \kr_{ij} [h_t]  )\delta f \, d\theta  \\
  & = \int \delta f \det _{\bar g} \kr_{ij} [h_t] \, d\theta   \eea 
  and 
  \bea \frac {d^2}{dt^2} A[\Omega_t] = \int \delta f (\det _{\bar g} \kr_{ij} [h_t] )\kb ^{ij}[h_t] \kr_{ij} [\delta f] \, d\theta  .\eea 
  
  At $t=0$, $A=\frac1n \la h,h \ra _h, $ $A' = \delta \la f , h \ra _h $
and $A''  = \delta^2 \la f , Lf \ra_h $.

\bea  \frac{d^2}{dt^2 }\bigg  \vert_{t=0 }  A^{\frac{1}{n}} [\Omega_t] &= \frac{1}{n A^{\frac{n-1}{n}}} \left [ A'' -\frac{n-1}{n} \frac{A'^2}{A} \right ]\\
 &=    \frac{\delta^2}{n A^{\frac{n-1}{n}}} \left[ \la f,Lf \ra_h - \frac{(n-1)\la f,h\ra_h ^2 }{\Vert h\Vert _h^2 }  \right] .\eea
 Let $f$ be an eigenfunction $Lf=\lambda f$ which is orthogonal to $h$ and $x_i$, for $i=1,\dots, n$. By the Brunn-Minkowski inequality, the second derivative above is strictly negative. Since $\la f ,h \ra _h=0$, we conclude $\lambda \Vert f \Vert_h^2 <0$. 
\end{proof}
\end{lemma}

We summarize two lemmas above and fix a spectral decomposition for each shrinker $h$ as follows.

\begin{definition}[Spectrum of $L$ in $L^2_h$] \label{def-spectrumL}For given shrinker $h$, $\{\varphi_i \}_{i=0}^\infty$ denotes an orthonormal basis 
 $L^2_h(\mathbb{S}^{n-1})$ which satisfies that
 \begin{enumerate}[leftmargin=1cm]
 \item the corresponding eigenvalues $\{\lambda_i\}_{i=0}^\infty $ are non-increasing with  
 \[n-1=\lambda_0 \ge \lambda _1=\dots =\lambda_n =0 > \lambda_{n+1}\ge \dots, \text{ and}  \] 
 \item 	$\varphi_0= h/\Vert h\Vert_h$, and $\varphi_i= x_i/\Vert x_i\Vert_h$ for $i=1,\ldots,n$. 	

 \end{enumerate}

 \end{definition}

We are interested in the difference between the solution and the asymptotic profile $Al^{\sigma}h$. 
From now on, let us define $w(s,\theta )$ by \be S(l,\theta  )= Al^{\sigma} h(\theta ) + w (s,\theta )\text{  where }s=\ln l.\ee  Next, we derive equations of $w$ when $S$ is a translator or blow-down translator. More precisely, we write the equation as a linearized equation for $w$ with additional non-linear error part. The following lemma is a consequence of a direct computation that uses $A l^\sigma h$ is a solution to \eqref{eq-blowdownsupport}.
\begin{lemma} \label{lemma-eqS} If $S(l,\theta  )$ is a blow-down translator, then 
\be \label{eq-17}0=w_{ss}+ (\frac{1-\sigma}{\alpha }-1)w_s + \sigma(1-\sigma) Lw + E_1(w) \ee 
 with
 \bea &\frac{E_1(w) }{e^{\sigma s}} = \frac{1-\sigma }{\alpha } \frac{w_s}{e^{\sigma s}}  h^{\frac1 \alpha -1} \left( \det_{\bar g}  \kr _{ij}[h+\frac{w}{Ae^{\sigma s}}] -\det_{\bar g }  \kr _{ij}[h] \right)  \\
&  + A\sigma (1-\sigma) h^{\frac1 \alpha}\left(\det_{\bar g}  \kr _{ij}[h+\frac{w}{Ae^{\sigma s}}] -\det_{\bar g }  \kr _{ij}[h] - (\det_{\bar g }  \kr _{ij}[h] )  \kb^{kl}[h] \kr _{kl} [\frac{w}{Ae^{\sigma s}}] \right) .\eea 
If $S(l,\theta )$ is a translator, then 
\be \label{eq-18} 0=w_{ss}+ (\frac{1-\sigma}{\alpha }-1)w_s + \sigma(1-\sigma) Lw + E_1(w)+E_2(w) \ee
with the same $E_1(w)$ as above and 
\bea \frac{E_2(w)}{e^{\sigma s}} = A \sigma &(1-\sigma) \left[ \left( 1+   \left (\sigma A e^{(\sigma-1)s }h+ {w_s}e^{-s} \right)^2 \right )^{\frac{n+2}2 -\frac1{2\alpha}}-1 \right] \\ \cdot &\left( h +\frac{w_s}{\sigma A e^{\sigma s}}\right)^{\frac 1\alpha } \det_{\bar g} \kr _{ij} [ h + \frac{w}{A e^{\sigma s}} ].\eea 


\end{lemma}

Note that $E_1(w)/e^{\sigma s}$ is at least quadratic in second order (or below) derivatives of $w/e^{\sigma s}$. Moreover, $E_2(w)/e^{\sigma s}$ has an extra exponential decay. Therefore, using an argument which uses Taylor's expansion with explicit remainder term or the mean value theorem, it is straight forward to obtain the following estimate below.  

\begin{lemma}[Estimate on $E_1$ and $E_2$] \label{lem-Eest}
Suppose  that $$\Vert e^{-\sigma s} w\Vert_{C^{2,\beta}_{s,1}}+ \Vert e^{-\sigma s} v\Vert_{C^{2,\beta}_{s,1}} \le \e_0  ,$$ for some $\e_0>0$. There exists $C=C(\alpha, h, \e_0)>0$ with the following significances: for all $s\in \mathbb{R}$, there holds
\bea  
&\Vert{e^{-\sigma s}}({E_1(w)- E_1(v) })\Vert_{C^{0,\beta }_{s,1}} \le C (\Vert e^{-\sigma s} w\Vert _{C^{2,\beta}_{s,1}} +\Vert e^{-\sigma s} v\Vert _{C^{2,\beta}_{s,1}})\Vert e^{-\sigma s} (w-v)\Vert _{C^{2,\beta}_{s,1}}   .\eea  
For $s\ge 1$, there holds \be \Vert{e^{-\sigma s}}{E_2(w) }\Vert_{C^{0,\beta }_{s,1}} \le C e^{2(\sigma-1)s}\ee 
and
\be \Vert{e^{-\sigma s}}({E_2(w)- E_2(v) })\Vert_{C^{0,\beta }_{s,1}} \le C e^{2(\sigma-1)} \Vert e^{-\sigma s} (w-v)\Vert _{C^{2,\beta}_{s,1}}  .  \ee 	
\end{lemma}

Note that $$\frac{1-\sigma}{\alpha }-1=  \frac{(n-1)-\alpha(n-2)}{1+\alpha(n-2)} \quad \text{  and }\quad \sigma (1-\sigma)=\frac{n\alpha (1-2\alpha)}{(1+\alpha(n-2))^2} .$$ From \eqref{eq-17}, let us define the linear operator 
\bea  \label{eq-calL}\mathcal{L} w 
&:= w_{ss}+ \frac{(n-1)-\alpha(n-2)}{1+\alpha(n-2)} w_s + \frac{n\alpha (1-2\alpha)}{(1+\alpha(n-2))^2} Lw\eea  which corresponds to the linearization of (blow-down) translator equation around the $\sigma^{-1}$-homogeneous solution $A l^{\sigma}h$.

\section{Main result and consequences} \label{section-main}
The linearized operator \eqref{eq-calL} has non-trivial kernel (Jacobi field) consisting of separation of variable solutions. By a direct computation, these are \[e^{\beta^+_i s} \varphi_i(\theta ),\quad \text{and} \quad e^{\beta^-_i s} \varphi_i(\theta ) \quad \text{ for } i=1 ,2,\ldots  \]
where the exponents are given $$\beta^\pm_i= -\frac{(n-1)-\alpha(n-2)}{2(1+\alpha(n-2))} \pm  \frac{\sqrt {((n-1)-\alpha(n-2))^2 -4n\alpha (1-2\alpha) \lambda_i}}{2(1+\alpha(n-2))}.$$
Note that $\beta^+_j+\beta^-_j= -\frac{(n-1)-\alpha(n-2)}{1+\alpha(n-2)}$ for all $j\ge0 $ and the first few $\beta^+_j$ are
\[\beta^+_0=\sigma-1 \quad \text{ and } \quad  \beta^+_i = 0 \quad \text{ for } i=1 ,\ldots,n.\]  
Our main theorem shows the Jacobi fields $e^{\beta^+_i s} \varphi_i(\theta)$ whose growth rate satisfy the bound $\beta^+_i < \sigma$ are effective ones that generate solutions. Let us denote the number of such Jacobi fields by 
\be  K := | \{i \in \mathcal{N}\cup\{0\} \,:\, \beta^+_j < \sigma  \}. \ee 
Note that the number $K$, exponents $\beta^\pm_i$ and eigenfunctions $\varphi_i$ are things that depend on $h$, $n$, and $\alpha$. Now we can state the main existence result. Our proof actually works for all $\alpha \in (\frac{1}{n+2},\frac12)$ as well. However, as pointed in Example \ref{example-sphere}, sphere $h\equiv1$ is the only shrinking soliton in this range with $K=n+1$. Since the radially symmetric solution with translations in $\mathbb{R}^{n+1}$ can replace the assertion of the theorem in this range, we only state this for $\alpha \in (0,\frac{1}{n+2})$. 

\begin{theorem}[Existence of translators]\label{thm-mainexistence} For $\alpha \in(0,\frac{1}{n+2})$, let $\Gamma \subset \mathbb{R}^{n}$ be a smooth $\frac{\alpha}{1-\alpha}$-shrinker represented by the support function $h(\theta)$ on $\theta \in \mathbb{S}^{n-1}$. There exists $K$-parameter family of $\alpha$-translators in $\mathbb{R}^{n+1}$,  say $\{ \Sigma_{\mathbf{y}}\} _{y \in \mathbb{R}^{K}}$, with the following properties: if $S_\mathbf{y}(l,\theta)$ denotes the support function of  $\Sigma_\mathbf{y}$ at level $l$. Then, 

\begin{enumerate}[leftmargin=1cm]
\item 	Each $\Sigma_\mathbf{y}$ is smooth strictly convex entire graph and the level set converges to the shrinker 
\[({Al^\sigma})^{-1} S(l,\theta) \longrightarrow h(\theta) \text{ in } C^{\infty}(\mathbb{S}^{n}) \text{ as } l\to \infty, \] 
\item For the solution at the origin $\Sigma_{\mathbf{0}}$, the support function  $S_{\mathbf{0}}$ is strongly asymptotic to $Al^{\sigma}h(\theta)$:  for all $\gamma> \sigma +2(\sigma-1)$,
\[S_{\mathbf{0}}(l,\theta) - Al^\sigma h(\theta) = o(l^{\gamma})\text{ as }l\to \infty .\] Next, for two vectors  $\mathbf{y}'=(0,\ldots,  0, y_{j+1}, \ldots , y_{K-1})$ and $\mathbf{y}''=(0,\ldots, b , y_{j+1}, \ldots, y_{K-1})$ in $\mathbb{R}^K$, the difference of corresponding solutions are dominated by a Jacobi field
\[S_{\mathbf{y}''} - S_{\mathbf{y}'} = b l^{\beta^+_j} \varphi_j + o(l^{\beta^+_j -\eps }) \text{ as } l\to \infty,\] for some $\e=\e(h,\alpha)>0$, and
\item The first $n+1$ parameters correspond to translations in $\mathbb{R}^{n+1}$ $$\Sigma_{\mathbf{y}+(a_0,a_1,\ldots,a_n,0,\ldots,0)} = \Sigma_{\mathbf{y}}+a_0c_0 e_{n+1} + \sum _{i=1}^n a_i c_i e_i   $$ where 
$c_0= -\frac{\Vert h \Vert _h}{A\sigma} $ and $c_i= \Vert x_i\Vert _h $ for $i=1,\ldots,n$.

\end{enumerate}

\end{theorem}

\begin{example}[Asymptotically round translator] \label{example-sphere}Suppose the shrinker $\Gamma$ is round ($h\equiv1$).  The spectrum of $L= \Delta + (n-1)$  on $\mathbb{S}^{n-1}$  can explicitly be enumerated as $$\tilde \lambda_\ell  = -\ell (\ell +n-2)+(n-1)\text{ for } \ell =0,1,\ldots $$
and each $\tilde \lambda_\ell$ has the multiplicity  $\binom{n+\ell -1 }{n-1}-\binom{n+\ell -3}{n-1}$. Here, the convention $\binom{n}{r}=0$ for $n <r$ is used. From this, we can explicitly compute the number of parameter $K$ as a function of $\alpha$. Let us define a decreasing sequence \[\alpha_{\ell} = \begin{cases}
 \begin{aligned}
 &1/2 && \text{ if } \ell =1\\ &1/({1-\lambda_\ell})={1}/({\ell ^2 +(n-2)\ell - (n-2)})	&& \text{ if } \ell \ge 2,
 \end{aligned}
	
 \end{cases}
\]
then 
\bea K= \frac{n+2\ell -1}{n+\ell -1} \binom{n+\ell -1}{n-1} &&\text{ for } \alpha \in [\alpha_{\ell+1},\alpha_\ell )
. \eea 
When $\alpha \in [\alpha_2, \alpha_1) = [1/(n+2), 1/2)$, $K= n+1$. i.e. solutions $\Sigma_\mathbf{y}$ in Theorem \ref{thm-mainexistence} are generated by translations in $\mathbb{R}^{n+1}$. In fact, by using the technique developed in \cite{CCK}, it is not difficult to show the converse that any translator which is asymptotically round equal to the rotationally symmetric translator upto a translation in $\mathbb{R}^{n+1}$. This type of classification question will be investigated in future research. When $n=2$, $\alpha_\ell = 1/\ell^2$ for $\ell \ge 2$ and $K=2\ell +1 $ if $\alpha \in [a_{\ell +1},a_\ell )$. The eigenfunctions on $\mathbb{S}^1$ are $1$, $\cos (k \theta)$, and $\sin (k \theta)$ for $k \ge1$. 
	
\end{example}

 If $\alpha > \frac1{n+2}$, then the round sphere $h\equiv1$ is the only $\frac{\alpha}{1-\alpha}$-shrinker in $\mathbb{R}^n$ \cite{brendle2017asymptotic}.  In the affine-critical case $\alpha=\frac{1}{n+2}$, classical result states that the convex paraboloids are the only $\alpha$-translators (in $\mathbb{R}^{n+1}$) and the ellipsoids are the only $\frac{\alpha}{1-\alpha}$-shrinkers (in $\mathbb{R}^n$). In the sub-affine-critical case $\alpha<\frac{1}{n+2}$, the existence of non-radial shrinker was shown by B. Andrews in \cite{andrews2000motion}. 
 
\begin{example}[Non-radial shrinkers] Andrews obtatiend the existence  of non-radial shrinker through a perturbation of radial solution using parabolic method. In our setting, this existence result can be stated as follows: 
\begin{theorem}[c.f. Theorem 18 \cite{andrews2000motion}] \label{theorem-andrews2000}Let $\Gamma$ be a proper subgroup of $SO(n)$ such that for every $z\in \mathbb{S}^{n-1}$, the orbit of $z$ under $\Gamma$ spans $\mathbb{R}^n$ (that is, the inclusion of $\Gamma$ in $SO(n)$ is an irreducible representation). Let $\lambda$ be the largest eigenvalue corresponding to a non-trivial $\Gamma$-invariant eigenfunction to $L=\Delta +(n-1)$. Then for $\frac{\alpha}{1-\alpha} \in (0,-\frac 1{\lambda})$, there exists a non-spherical, $\Gamma$-symmetric, smooth, strictly convex $\frac{\alpha}{1-\alpha} $-shrinker in $\mathbb{R}^n$. 
\end{theorem}
As an application of Theorem \ref{theorem-andrews2000}, when $n=3$, there is a shrinking surfaces in $\mathbb{R}^3$ with tetrahedral symmetry for $\frac{\alpha}{1-\alpha} \in (0,\frac 1{10})$; one with octahedral symmetry for $\frac{\alpha}{1-\alpha} \in (0, \frac 1{18})$; one with icosahedral symmetry for $\frac{\alpha}{1-\alpha} \in (0,\frac{1}{40})$ \cite[Corollary 21]{andrews2000motion}. Similarly, for higher dimension, there is a shrinking soliton in $\mathbb{R}^n$ with the symmetry of $n+1$-simplex for $\frac{\alpha}{1-\alpha} \in (0, \frac{1}{2n+4})$; one with the symmetry of a regular hypercube for $\frac{\alpha}{1-\alpha} \in (0,\frac{1}{3n+9})$ \cite[Corollary 22]{andrews2000motion}. Theorem \ref{thm-mainexistence} applies for each non-radial shrinker and shows the existence of translating solitons which are asymptotically non-rotational.

Note  $\frac{\alpha}{1-\alpha} = -\frac{1}{\lambda}$ $\Leftrightarrow$ $\alpha = \frac{1}{1-\lambda}$ and it is not a huge coincidence that $\frac{1}{1-\lambda}$ also appears in the definition of $\alpha _\ell$ in Example \ref{example-sphere}. One could understand this from an aspect of bifurcation: let $\Gamma$ be a subgroup of $SO(2)$ satisfying the conditions of Theorem \ref{theorem-andrews2000} and $\varphi$ and $\lambda$ be corresponding $\Gamma$-invariant eigenfunction and eigenvalue. At $\alpha=\frac{1}{1-\lambda}$, the operator \eqref{eq-calL} admits a Jacobi field(kernel) $w(s,\theta)=\varphi(\theta) e^{\sigma s}$. When $\alpha$ becomes smaller than $\frac{1}{1-\lambda}$, this Jacobi field creates two options: either this adds another parameter in the existence of asymptotically round solution whose growth rate is close but strictly smaller than $\sigma$, or this creates a new translating soliton which is asymptotic to a $\Gamma$-invariant shrinking soliton at infinity. It is expected that such a shrinking soliton is not radial, but converges to radial one as $\alpha$ approaches to $\frac{1}{1-\lambda}$ from the left. See \cite{andrews2003classification} where this property is shown for shrinking curves in $\mathbb{R}^2$.
	
\end{example}

In fact, when $n=2$, the $\frac{\alpha}{1-\alpha}$-shrinkers in $\mathbb{R}^2$ and $\alpha$-translators in $\mathbb{R}^3$ are completely classified for $\alpha\le \frac{1}{4}=\frac{1}{n+2}$. Except the classical affine critical case $\alpha=\frac{1}{4}$, the classification for shrinker is shown by Andrews \cite{andrews2003classification} and the classification for translator is recently shown by the author, K. Choi, and S. Kim in \cite{CCK}. \cite{andrews2003classification} shows, for each $k\ge3$, there is a $k$-fold symmetric $\frac{\alpha}{1-\alpha}$-shrinking curve in $\mathbb{R}^2$ for $\alpha  \in (0,k^{-2})$ (this corresponds to eigenfunction $\cos k\theta$ on $\mathbb{S}^1$) and there is no other non-radial shrinking curve for $\alpha<\frac{1}{4}$ except their rotations. In \cite{CCK}, the existence of translator is shown as in Theorem \ref{thm-mainexistence} and arbitrary translator is shown to be exactly one of constructed solutions. We expect that a similar classification should hold for higher dimensions and conjecture this as follows.
\begin{conjecture}[Sub-affine-critical translators] \label{conjecture} For $\alpha<\frac{1}{n+2}$, every $\alpha$-translator in $\mathbb{R}^{n+1}$ must be asymptotic to a $\frac{\alpha}{1-\alpha}$-shrinker in $\mathbb{R}^n$. Moreover, it is equal to one of solutions constructed in Theorem \ref{thm-mainexistence}. 
	
\end{conjecture}

If a translator is asymptotic to a shrinker, it necessarily has the same growth rate in all directions. We expect such a homogeneous growth rate only in sub-affine-critical range and, in super-affine-critical range, there exist translators which have inhomogeneous growth rates in different directions and hence the (convex) level set becomes degenerate as it gets higher. Finally, for $\alpha>1/2$, the translators can only exist as graphs on bounded domains and the classification is shown by J. Urbas in \cite{U98GCFsoliton,Urbas1988}.

\subsection{Proof of Theorem \ref{thm-mainexistence}}
In this section, the H\"older exponent $\beta$ is assumed to belong $\beta\in(0,1)$ and all constants in estimates may depend on $\beta$ (although we do not specify this).  From Lemma \ref{lemma-eqS} and \eqref{eq-calL}, the translator equation is equivalent to find a solution $w(s,\theta)$ to the equation  
\be\label{eq-wtranslator} \mathcal{L}w =E(w) \text{ where }E(w)=-E_1(w) -E_2(w).\ee In the first part, we solve \eqref{eq-wtranslator} on exterior domains $B_R^c= [R,\infty)\times \mathbb{S}^{n-1}$ where $R>1$ is sufficiently large. Finally, we use the convention $\beta^+_{-1}:=\beta^-_0$. The proof follows similar lines of argument from \cite[Section3]{CCK}.  For example, the next Lemma \ref{lem-ex-linear}, Proposition \ref{prop-exist-exterior0}, and Proposition \ref{prop-exist-exterior'} follows by the same proof of Lemma 3.4, Proposition 3.6 and Proposition 3.7 in \cite{CCK}. For those proofs, we will give brief sketches instead of details.

\begin{lemma}[Solvability of linear equation]\label{lem-ex-linear}
 Let  $\gamma$ be a constant such that  $ \beta^+_{m-1}< \gamma  <\beta^+_m$  for   some $m\in \{0,\ldots, K-1\}$  (with  the convention $\beta^+_{-1}:= \beta^-_0$).    
For each  $g$ in $C^{0,\beta,\gamma}_R$,    there exists $w\in C^{2,\beta,\gamma}_R $ which solves the linear equation 
\be \label{eq-w-linear}
\begin{cases}\begin{aligned}   &\mathcal{L}w   =g\quad && \hbox{ in  } B_R^c \\ & w=0\quad  && \hbox{ on } \partial B_R^c .  \end{aligned} \end{cases} \ee 
More precisely, this $w$ is explicitly written as    
\be\label{eq-w-decomp}
w=\sum_{j=0}^\infty  w_j(s) \varphi_j(\theta),\ee
 where 
\begin{align} 
  w_j (s)&= -e^{\beta^-_j s} \int_R^s e^{(\beta^+_j-\beta^-_j)r }\int_r^\infty e^{-\beta^+_j t}g_j(t) dt dr \quad \text{ for  }j \ge m,  \label{eq-wj}\\
  w_j (s)&= e^{\beta^-_j s} \int_R^s e^{(\beta^+_j-\beta^-_j)r }\int_R^r e^{-\beta^+_j t}g_j(t) dt dr \quad  \text{ for }0\le j <m \label{eq-wj2}
\end{align}
with  $g_j(s):= \la  g(s,\cdot),\varphi_j \ra _h$. 
Moreover, there holds 
\bea \label{eq-w-est}
 \Vert w \Vert_{C_R^{2,\beta,\gamma}} \le C \Vert g \Vert_{C_R^{0,\beta,\gamma} }
 \eea  
 for some constant $C>0$ independent of $R>1$.  

\begin{proof}The proof immediately follows by the linear elliptic regularity theory once we establish the following $L^2$-estimate with weight:
 \bea\label{eq-L-2-linear-sol}
  e^{-\gamma s}\Vert w(s,\cdot)\Vert_{L^2_h} =  e^{-\gamma s}  \Big(\sum_{j=0}^\infty {w_j^2(s)} \Big)^{1/2}   \le C \Vert g \Vert_{C_R^{0,0,\gamma}}\quad\hbox{for all $s\geq R$}.
  \eea   
This estimate follows from an integral estimate by working with \eqref{eq-wj} and \eqref{eq-wj2}. (See \cite[Claim 3.1]{CCK}.) 
\end{proof}
   \end{lemma} 
  For given $R>1$ and $\gamma \in (\beta^+_{m-1},\beta^+_m)$, let us denote the inverse of $\mathcal{L}$, which assigns $w$ for each given $g$ as in \eqref{eq-wj} and \eqref{eq-wj2} by $\mathcal{H}_{R,\gamma}:C^{0,\beta,\gamma}_R \to C^{2,\beta,\gamma}_R$. We simply write it $\mathcal{H}$ if there is no other confusion. In the next proposition, we  show if a translator (on exterior domain) and a Jacobi field are given, then we can construct another translator (on exterior domain) whose difference with given translator is dominated by the Jacobi field. The proof is based on the fixed point argument which uses Lemma \ref{lem-ex-linear}. 
  
  \begin{prop}\label{prop-exist-exterior0} Suppose $w \in C_{R_0}^{2,\beta,\gamma_1} $ solves \eqref{eq-wtranslator} on $B_{R_0}^c$  for $R_0>1$ and  $\gamma_1 \in(0,\sigma)$ and $u_0 \in C^{2,\beta,\gamma_2}_{R_0}  $  solves  $\mathcal{L} u_0 =0$ on $B_{R_0}^c$     for     $\beta_{m-1}^+<\gamma_2  \leq   \beta_m^+$  with  $m\in\{0,1,\cdots,K-1\}$.
   	Then there exists $u\in  C^{2,\beta,\gamma_2-\e} _R$ with  $R>R_0$ and $\e>0$ such that $w+u_0+u$  solves \eqref{eq-wtranslator} on $B_R^c$ with  zero boundary condition. 
  Here,  constants $R>R_0$ and    $\e>0$ may depend on $ \alpha$, $h$,  $ \gamma_1,$ $  \gamma_2 $,   $ \beta^+_{m-1}$, $ \beta^+_{m}$, $\Vert w\Vert_{C_{R_0}^{2,\beta,\gamma_1}} $ and $\Vert u_0\Vert_{C_{R_0}^{2,\beta,\gamma_2}}$. 

\begin{proof} For $R>R_0$ and $\gamma=\gamma_2-\eps \in (\beta^+_{m-1},\beta^+_m)$ to be chosen below, let us consider the inverse operator $\mathcal{H}:C^{0,\beta,\gamma}_R \to C^{2,\beta,\gamma}_R$. We define $u_1=\mathcal{H}(E(w+u_0)-E(w))$ and similarly define sequence of functions $u_{k+1}= \mathcal{H}(E(w+\sum_{i=0}^k u_i)-E(w+\sum_{i=0}^{k-1} u_i)) $ for $k\ge1$. Our goal is to show that if we choose sufficiently large $R$ and small $\eps>0$, $\Vert u_i\Vert _{C^{2,\beta,\gamma}_R}$ geometrically converges to $0$ as $i\to 0$. In that case, we denote $u=\sum_{i=1}^\infty u_i$ and it is straightforward to see that $w+u_0+u$ solves \eqref{eq-wtranslator}.

First, replace $\gamma_1$ by $\max (\gamma_1,\gamma_2)$. In view of Lemma \ref{lem-Eest} and \eqref{eq-w-est}, if we choose $$\eps =  \frac 12 \min(\sigma-\gamma_1, 2(\sigma-1), \gamma_2-\beta^+_{m-1}) $$  then there holds an estimate
\[\Vert u_{k+1} \Vert_{C^{2,\beta,\gamma}_R} \le Ce^{-\eps R}  ( 1+ \Vert w\Vert_{C^{2,\beta,\gamma_1}_R}+\sum_{i=0}^k \Vert u_i\Vert_{C^{2,\beta,\gamma_1}_R}  )\Vert u_k \Vert _{C^{2,\beta,\gamma_2}_R}.\]
Using above estimate and trivial inequality $\Vert v\Vert_{C^{2,\beta,\gamma}_R }\le \Vert v\Vert_{C^{2,\beta,\gamma_2}_R}$, we conclude the geometric convergence of $\Vert u\Vert_{C^{2,\beta,\gamma}_R}$ for large $R$.  

\end{proof}
   \end{prop}

Next proposition shows we can perturb  given translator (on exterior domain) so that the level set is homothetic to the shrinker at the boundary with certain gradient condition. Later we will glue super and sub solutions along this boundary. Those super and sub solutions will have level sets that are homothetic to the shrinker.

\begin{prop}\label{prop-exist-exterior'}   Let $w \in C_{R_0}^{2,\beta,\gamma_1} $ be a solution to \eqref{eq-wtranslator} on $B_{R_0}^c$ for $R_0>1$ and    $\gamma_1 \in (0,\sigma)$. Let  $\gamma_2$ be a constant such that  $\gamma_1 <{\gamma_2  < \sigma} $. Then, there is a constant $R_1 >  R_0$ (depending on $\Vert w \Vert _{C^{2,\beta,\gamma_1}_{R_0}}$, $\gamma_1$ and  $\gamma_2$) such that for each $R\ge R_1$ there exists   $u \in C^{2,\beta, \gamma_0  }_{R}$, $\gamma_0=  -\frac{(n-1)-\alpha(n-2)}{2(1+\alpha(n-2))}$, such that $w+u$ solves \eqref{eq-wtranslator} on $B_R^c$ with  the following    boundary conditions: 
\be\label{eq-exist-ex-bd1'}
\left\{\ba
& 
 \frac{w+u}h = -e^{\gamma_2 R } \quad \text{ on } \p B_R,\\
&  \frac{\p }{\p s } \left(\frac{w+u}{h}\right) = -\beta^-_0 e^{\gamma_2 R}(1+o(1))\,\, \text{ on } \, \,\p B_R.
\ea\right.
\ee 
 Here, $  \frac{\p }{\p s }  (\frac{w+u}{h} ) = \sigma e^{\gamma_2 R}(1+o(1)) $ on $\p B_R$ means that for sufficiently large $R$, the corresponding solution  $u=u_R $ defined on $B_R^c$ satisfies  
 \bea\left. \left \Vert   \frac{\p }{\p s }\right\vert_{s=R} \left( \frac{ g(s,\cdot)+u_R(s,\cdot)}{h(\cdot) }\right)+ \beta^-_0 e^{\gamma_2 R}\right \Vert_{L^\infty(\mathbb{S}^1)} = o (e^{\gamma_2 R}) \quad\text{ as } R\to \infty.   \eea 
Note that $\gamma_0$ is chosen so to satisfy $\beta^-_0<\gamma_0<\beta^+_0=\sigma-1$. 

Similarly, for  each $R\ge R_1$  with  some large $R_1 >  R_0$,      there exists   a solution   $u \in C^{2,\beta,\gamma_0  }_R$ such that $w+u$ solves \eqref{eq-wtranslator} with  the following    boundary conditions:  
\be\label{eq-exist-ex-bd2'}
\left\{\ba
&
\frac{g+u}h   = e^{\gamma_2 R } \text{ on } \p B_R,\\
&  \frac{\p }{\p s } \left(\frac{g+u}{h}\right) = \beta^-_0 e^{\gamma_2R}(1+o(1))\,\, \text{ on } \, \,\p B_R.
\ea\right.
\ee
	
\begin{proof} We give a proof of the first part as the other part is similar. For $R\ge R_0$, let us consider $\hat g_R\in C^{2,\beta,\beta^-_0}_R$ which satisfies 
\be\left\{
\ba & \mathcal{L}(\hat g_R) =0&&\text{ in } B_R^c ,\\
&\hat g_R = -g &&\text{ on }\p B_R.
\ea\right.
\ee 
More specifically, if $g(R,\cdot) =\displaystyle \sum_{i=0}^\infty c_i  \varphi_i(\cdot) $, we may choose \bea  \hat g_R(s,\theta) = - \sum _{j=0}^\infty c_j e^{\beta^-_j (s-R)} \varphi_j(\theta)\quad \text{ in } B_R^c \eea and this $\hat g_R$ belongs to $C^{2,\beta,\beta^-_0}_R$ by the linear elliptic regularity theory. 
For $R\ge R_0$ to be chosen later, if we consider  
\bea  u_0(s,\theta) := \hat  g_R (s,\theta)- e^{\gamma_2R} e^{\beta^-_0 (s-R)}h(\theta),\eea then it is direct to check \be\label{eq-u_0cases}
 \begin{cases}
  \begin{aligned} &\hat L (u_0) = 0 && \text{ on } B_R^c , \\  
  &\frac{g+u_0}{h} = - e^{\gamma_2R } &&\text{ on } \p B_R, \\
 &\frac{\p}{\p s} \left(\frac{g+u_0}{h} \right)= -\beta^-_0 e^{\gamma_2 R} (1+o(1)) &&\text{ on } \p B_R.
 \end{aligned}
 \end{cases} 
 \ee

In Lemma \ref{lem-ex-linear}, let us choose $\gamma$ as $\gamma_0= -\frac{(n-1)-\alpha(n-2)}{2(1+\alpha(n-2))}\in (\beta^-_0,\beta^+_0),$ we obtain $\mathcal{H}_{R,\gamma_0}=\mathcal{H}:C^{0,\beta,\gamma_0}_R \to C^{2,\beta,\gamma_0}_R$. We define $u_1=\mathcal{H}(E(w+u_0)-E(w))$ and $u_{k+1}= \mathcal{H}(E(w+\sum_{i=0}^k u_i)-E(w+\sum_{i=0}^{k-1} u_i)) $ for $k\ge1$. By a similar argument as in the proof of Proposition \ref{prop-exist-exterior0}, we obtain that there is $\eps>0$ and $R_1\ge R_0$ such that for all $R$, \[\Vert u_{k+1}\Vert_{C^{2,\beta,\gamma_0}_R} \le \frac12 e^{-\eps R}\Vert u_k \Vert_{C^{2,\beta,\gamma_0}_R}. \]
It follows that $w+u_0 +\sum_{i=1}^\infty u_i = w +u$ solves \eqref{eq-wtranslator}. Moreover, \be \label{eq-u-u_0cases}\begin{cases} \begin{aligned}  & \sum_{i=1}^\infty u_i =0 &&\text{ on } \p B_R, \\
 &\frac{\p}{\p s}  \sum_{i=1}^\infty u_i =  O\left(e^{(  \gamma_2-\e)R}\right) &&\text{ on } \p B_R,
 \end{aligned} \end{cases} \ee  
where we used $\Vert \sum_{i=1}^\infty u_i \Vert _{C^{2,\beta,\gamma_0}_R} \le e^{-\eps R} \Vert u_0\Vert _{C^{2,\beta,\gamma_0}_R} $ in the second estimate. 
\end{proof}

\end{prop}

Let $f_M(l)$, on $l\ge0$, be the unique radial solution to the equation 
\be\label{eq-translatorsupportM} S_{ll}+(1+M S_l ^{2})^{\frac{n+2}{2} - \frac{1}{2\alpha}} S_l^{\frac1 \alpha}{\det_{\bar g}(  \kr _{ij}[S]) }   =0\ee
with the boundary condition $f_M(0)=0$ and $f_m'(0)=\infty$. Then $f_M$ has the asymptotics 
\be\label{eq-fMasymptotics} f_M (l)= Al^{\sigma} + C_\alpha M l^{\sigma +2(\sigma-1)} (1+o(1)) \text{ as } l\to \infty \ee 
where $C_\alpha>0$ if $\alpha <\frac{1}{n+2}$ and $C_\alpha<0$ if $\alpha >\frac{1}{n+2}$.
\begin{lemma}There is $M_1>0$ (which is large if $\alpha <\frac{1}{n+2}$ and small if $\alpha >\frac{1}{n+2}$) such that the convex hypersurface whose support function at height $l$ is represented by  $U_1(l,\theta)= f_{M_1}(l) h(\theta)$ is a viscosity sub solution to the equation \eqref{eq-translatorsupport} and there is  $M_2>0$ (small if $\alpha >\frac{1}{n+2}$ and large if $\alpha >\frac{1}{n+2}$) such that $U_2(l,\theta)= f_{M_2}(l)h(\theta)$ is a viscosity super solution to the equation \eqref{eq-translatorsupport}. 

\begin{proof} We will prove for the sub solution case as the other super solution is similar. Let us consider the smooth part of $S_1$ where $l>0$. If we plug $S_1$ into \eqref{eq-translatorsupport}, then we obtain $S_1$ is subsolution on $l>0$ if \[f_{M_1}''+(1+ (h f'_{M_1})^2 )f_{M_1}^{\frac1\alpha} f^{n-1} \ge 0. \]
	In view of \eqref{eq-translatorsupportM}, we obtain that the inequality is satisfied if $M _1\ge \sup _{\theta \in \mathbb{S}^{n-1}}h^2(\theta) $.
	
	Next, we are left to show the hypersurface solves the equation in the viscosity sense at the origin. Let $x_{n+1}=u(x)$ be the graphical representation of the hypersurface and let $v(x)$ be a smooth function which touches $u(x)$ from above at $x=0$. Note that $Du(0)=Dv(0)=0$. We want to show $\det D^2v \ge (1+|Dv|^2)^{\frac{n+2}2-\frac{1}{2\alpha}} = 1$. Let $A_v(l)$ be the area of the region enclosed by the level set of $v$ at height $l$. 
	
	From the fact that $D ^2v(0)$ gives the opening of approximating paraboloid, $A_v(l)= \frac{|B_1|(2l)^{n/2}}{(\det D^2 v(0))^{1/2}}(1+o(1))$ as $l\to 0^+$. Meanwhile, from the fact that the level sets of $u$ are homothetic to shrinker represented by $h(\theta)$ and the inverse function of $f_{M_1}(l)$ represent the radial solution to $\det D^2 u = (M+|Du|^2)^{\frac{n+2}2-\frac{1}{2\alpha}}$ which satisfies $u(0)=0$, we infer $A_u(l)={C_{h}{M_1^{\frac{1}{4\alpha}- \frac{n+2}{4} }} l^{n/2}}(1+o(1))$ as $l\to 0^+$. From the inequality $A_v(l) \le A_u(l)$ which is due to $u\le v$, we conclude $\det D^2 v (0) \ge 1$ for sufficiently large $M_1$ if $\alpha <\frac{1}{n+2}$ and sufficiently small $M_1$ if $\alpha >\frac{1}{n+2}$. 
\end{proof}
\end{lemma}

 Now we assert that whenever an exterior solution is given we may extend it into a complete translator in exchange for a modification of exterior solution by a function with very fast decay.

\begin{prop} \label{prop-completetranslator} Let $w \in C^{2,\beta,\gamma_1}_{R_0}$ represent an exterior translator by solving \eqref{eq-wtranslator}. Then there exists a smooth complete translator $\Sigma \subset \mathbb{R}^{n+1}$ which is asymptotic to $w$ for the following sense: if the support function of $\Sigma$ at level $l$ is denoted by \[S(l,\theta) = Al^{\sigma } h + w(s,\theta)+ \hat w(s,\theta)\] with $s= \ln l$, then $\hat w$ belongs to $C^{2,\beta,\gamma_0}_R$, $\gamma_0= -\frac{(n-1)-\alpha(n-2)}{2(1+\alpha(n-2))}$, and some $R>R_0$.  

\begin{proof} By Proposition \ref{prop-exist-exterior'}, for $\gamma_2 \in (\gamma_1, \sigma)$,  there is $R_1\ge R_0$ such that for each $R\ge R_1$, there are (incomplete) translators on exterior domain $B_R^c$, namely $V_1(s,\theta)$ for $i=1$, $2$, which is of form $V_i(s,\theta)= Ae^{\sigma s} h + w(s,\theta)+w_i(s,\theta)$ with $w_i \in C^{2,\beta,\gamma_0}_R$, 
\be\label{eq-exist-ex-bd1-glo}
\left\{\ba
&   V_1 /h  =        A e^{\sigma R}  -  e^{\gamma_2 R}\qquad \hbox{on $\p B_R$},  \\
& \frac{\p }{\p s}   \left(  V_1/{h}\right)=    A\sigma e^{ \sigma R} -\beta^-_0  e^{\gamma_2 R} (1+o(1)) \qquad \hbox{on $\p B_R$},  
\ea\right.
\ee
and 
\be\label{eq-exist-ex-bd2-glo}
\left\{\ba
&   V_2 /h  =         A e^{ \sigma R} + e^{\gamma_2 R} \qquad \hbox{on $\p B_R$},   \\
&  \frac{\p }{\p s}   \left(  V_2/{h}\right)= A\sigma    e^{ \sigma R}  +\beta^-_0e^{\gamma_2 R} (1+o(1))\qquad \hbox{on $\p B_R$}.
\ea\right.
\ee  

We glue $S_i(l,\theta)= f_{M_i}(l)h(\theta)$ to $V_i(s,\theta)$ along $s=R$ in the following way: let us fix $\gamma_2$ so that $ \max  (\gamma_1, \sigma+2(\sigma-1))<\gamma_2<\sigma $. For    a large constant   $R ( \geq R_1)$   to be determined later,     let us   define
\bea S_1(l,\theta) := 
\begin{cases}
 \begin{aligned} & V_1(\ln l,\theta) && \text{ for  } \,\, l \ge e^R, \\
 &U_1(l-l_1,\theta) && \text{ for  }\,\, l_1\le l\le e^R,
 \end{aligned}
  \end{cases} 
  \eea 
  where  $l_1$ is the unique number such that    $S_1$  is continuous along  $l= e^R$.
  Similarly, we define
   \bea
    S_1(l,\theta) := 
    \begin{cases} 
    \begin{aligned} & V_2(\ln l,\theta) && \text{ for }\,\, l \ge e^R, \\
 &U_2(l-l_2,\theta) && \text{ for }\,\, l_2\le l\le e^R ,
 \end{aligned}
  \end{cases}
   \eea 
  with     the unique number    $l_2$   to make  $S_2$        continuous along  $l= e^R$. 

By the choice of $\gamma_2$, the asymptotics of inner barriers in \eqref{eq-fMasymptotics}, and the asymptotics of exterior translators in \eqref{eq-exist-ex-bd1-glo} and \eqref{eq-exist-ex-bd2-glo}, if $R$ is sufficiently large, there hold \bea \label{eq-glob-barrier-angle-1}
 \lim_{\e \to 0^-} \p_l S_1(e^R+\e ,\theta) <\lim_{\e \to 0^+} \p_l S_1(e^R+\e ,\theta)\quad \text{ for all }\theta \in \mathbb{S}^1   \eea 
 and   
 \bea \label{eq-glob-barrier-angle-2}
  \lim_{\e \to 0^-} \p_l S_2(e^R+\e ,\theta) > \lim_{\e \to 0^+} \p_l S_2(e^R+\e ,\theta)\quad \text{ for all }\theta \in \mathbb{S}^1 .  \eea 
This implies that $S_1$ and $S_2$ are global viscosity sub solution and super solution, respectively. Moreover $S_1 \le S_2$ on their common domain. Indeed, by the asymptotics of $V_i$ which implies $V_1-V_2=o(l^{\gamma_0})$ with $\gamma_0 < \sigma-1=\beta^+_0$, for each small $\e>0$, the comparison principle implies $S_1(l-\eps ,\theta) \le S_2(l,\theta)$. Then we send $\eps\to 0^+$ to conclude the inequality. 

Finally, we use global barriers $S_1$ and $S_2$ to construct a smooth translator $\Sigma\subset \mathbb{R}^{n+1}$ whose support function satisfies 
\[S_1(l,\theta)\le S(l,\theta) \le S_2(l,\theta).\] The convex hypersurace represented by $S_i(l,\theta)$ can be expressed as the graph of entire convex function $\{x_{n+1}=u_i(x)\}$. Then $u_1$ is a viscosity super solution to the equation \be \label{eq-graphtranslator} \det D^2 u = (1+|Du|^2)^{ \frac{n+2}{2} - \frac{1}{2\alpha}}\ee  and $u_2$ is a viscosity sub solution. Moreover, $u_2\le u_1$ on $\mathbb{R}^n$. By using one of $u_i$ as barriers and using one of $u_i$ as boundary condition, we may obtain unique translators on bounded domains. Then using an interior regularity estimate, we may pass to a subsequence of solutions on bounded domains to obtain a smooth translator defined on $\mathbb{R}^n$.  This solution $u$ satisfies $u_2\le u \le u$ and we define its graph by $\Sigma$.

\end{proof} 

 \end{prop}

\begin{remark} \label{remark-zerotranslator} Note that $w\equiv 0$ is not a solution to \eqref{eq-wtranslator} since $E_2(w)\neq0$. We obtain an exterior solution which is asymptotic to $0$ using the argument of Proposition \ref{prop-exist-exterior0} and then extend this to a complete  $\Sigma _{\mathbf{0}}$ by Proposition \ref{prop-completetranslator}.  By fixing any $\gamma \in (3\sigma -2, \sigma)$, we observe $E(0) \in C^{0,\beta,\gamma}_R$ and hence define similarly as in Proposition \ref{prop-exist-exterior0}: $u_0=\mathcal{H}(E(0))$, $u_1=\mathcal{H}(E(u_0)-E(0))$, and $u_{k+1}=\mathcal{H}(E(\sum_{i=0}^{k} u_i) - E(\sum_{i=0}^{k-1} u_i)).$ Then we could show for large $R$, $\Vert u_i\Vert_{C^{2,\beta,\gamma}_R}$ geometrically converges to $0$ and hence $\sum_{i=0}^\infty u_i$ is a translator on $B_R^c$. Then we apply Proposition \ref{prop-completetranslator} to obtain a complete translator asymptotic $\Sigma_{\mathbf{0}}$. 
\end{remark}

Using the ingredients above, we can now give the proof of main theorem. 

\begin{proof}[Proof of Theorem \ref{thm-mainexistence}] The proof follows by a successive applications of Proposition \ref{prop-exist-exterior0} and Proposition \ref{prop-completetranslator}. In this proof $w_{\mathbf{y}}(s,\theta)$ denotes the function defined by $S_{\mathbf{y}}(l,\theta) = Al^{\sigma}h(\theta)+ w_{\mathbf{y}}(s,\theta)$ with $s=\ln l$. 

First, by Remark \ref{remark-zerotranslator} and Proposition \ref{prop-completetranslator}, there exists a translator we denote by $\Sigma_{\mathbf{0}}$. Next, we construct solutions by allowing non-zero components from the last one. For any given $\mathbf{y}' =(\mathbf{0}_{K-1}, a)\in\mathbb{R}^K$ with $a\in\mathbb{R}$, plugging $w_{\mathbf{0}}$ and $a \varphi_{K-1} e^{\beta^+_{K-1}s}$ into $w$ and $u_0$ in Proposition  \ref{prop-exist-exterior0}, repectively (here we also choose $\gamma_2=\beta^+_{K-1}$), we obtain an exterior translator and then by Proposition \ref{prop-completetranslator} this extends to a complete translator which we denote by $\Sigma_{\mathbf{y}'}$. Similarly, given that translators are assigned for all parameters of form $\mathbf{y}=(\mathbf{0}_j,\mathbf{a}_{K-j}) $ with $\mathbf{a}\in \mathbb{R}^{K-j}$, for $\mathbf{y}'=(\mathbf{0}_{j-1},a,\mathbf{a}_{K-i})$, there exists a complete translator by Proposition \ref{prop-exist-exterior0} (plugging $w_{(\mathbf{0}_j,\mathbf{a}_{K-j})}$ and $a \varphi_{j-1} e^{ \beta^+_{j-1}s}$ into $w$ and $u_0$ with $\gamma_2= \beta^+_{j-1}$) followed by Proposition \ref{prop-completetranslator}. In this way, we can construct solutions $\Sigma_{\mathbf{y}}$ which satisfy the conditions (1) and (2) in Theorem \ref{thm-mainexistence}. 

Finally, given that translators $\Sigma_{\mathbf{y}}$ are assigned for all parameters of form $\mathbf{y}= (\mathbf{0}_{n+1}, \mathbf{a}_{K-n-1})$, let us define the solution for all $\mathbf{y}\in\mathbb{R}^{n+1}$ by following the instruction given in (3). By the choice of $\varphi_0,\ldots \varphi_n$ in Definition \ref{def-spectrumL}, and the fact $\beta^+_1=\beta^+_{n}= 0$ and $\beta^+_0=\sigma-1$, it is straight forward to observe that these solutions satisfy condition (2) as well.

\end{proof}

 \bibliography{GCF-ref.bib}
\bibliographystyle{abbrv}

\end{document}